\documentclass[a4paper,11pt]{amsart}
\usepackage[utf8]{inputenc}
\usepackage{csquotes}
\usepackage{amsmath, amssymb, amsthm, xcolor,enumerate, enumitem}
\usepackage[english]{babel}
\usepackage{tikz-cd}
\usepackage[all]{xy}
\usepackage{microtype}
\usepackage[colorlinks=true, citecolor=blue, linkcolor=blue, bookmarks=true]{hyperref}
\usepackage{bbm}
\usepackage{graphicx}
\usepackage{adjustbox}
\usepackage{comment}
\usepackage{mathtools}

\newcounter{theorem}
\newtheorem{theorem}[theorem]{Theorem}
\newtheorem{lemma}[theorem]{Lemma}
\newtheorem{prop}[theorem]{Proposition}
\newtheorem{cor}[theorem]{Corollary}

\theoremstyle{definition}
\newtheorem{defn}[theorem]{Definition}
\newtheorem{notation}[theorem]{Notation}

\theoremstyle{remark}
\newtheorem*{remark*}{Remark}
\newtheorem{rmk}[theorem]{Remark}

\numberwithin{equation}{section}
\allowdisplaybreaks

\newcommand{\cstar}[0]{\ensuremath{\mathrm{C}^*}}
\newcommand{\id}{\mathrm{id}}

\newcommand{\dimrokc}{\dim_{\mathrm{Rok}}^{\mathrm{c}}}
\newcommand{\Aut}{\mathrm{Aut}}

\title{On a Rokhlin property for abelian group actions on \cstar-algebras}

\author[J.\ Christensen]{Johannes Christensen}
\address{\hskip-\parindent Department of mathematics, Aarhus University, Ny Munkegade 118, 1530, 8000 Aarhus C, Denmark}
\email{johannes@math.au.dk}
\author[R.\ Neagu]{Robert Neagu}
\author[G.\ Szabó]{Gábor Szabó}
\address{\hskip-\parindent Department of mathematics, KU Leuven, Celestijnenlaan 200B, box 2400 \linebreak 3001 Leuven, Belgium.}
\email{robert.neagu@kuleuven.be}
\email{gabor.szabo@kuleuven.be}

\thanks{RN and GS funded by the European Union. Views and opinions expressed are however those of the authors only and do not necessarily reflect those of the European Union or the European Research Council. Neither the EU nor the ERC can be held responsible for them.}

\begin{document}

\begin{abstract}
In this article, we study the so-called abelian Rokhlin property for actions of locally compact, abelian groups on \cstar-algebras. 
We propose a unifying framework for obtaining various duality results related to this property.
The abelian Rokhlin property coincides with the known Rokhlin property for actions by the reals (i.e., flows), but is not identical to the known Rokhlin property in general.
The main duality result we obtain is a generalisation of a duality for flows proved by Kishimoto in the case of Kirchberg algebras.
We consider also a slight weakening of the abelian Rokhlin property, which allows us to show that all traces on the crossed product \cstar-algebra are canonically induced from invariant traces on the the coefficient \cstar-algebra.
\end{abstract}

\maketitle

\numberwithin{theorem}{section}	

\section*{Introduction}
\renewcommand*{\thetheorem}{\Alph{theorem}}

One of the landmark results in the study of locally compact, abelian groups is the Pontryagin duality.
In simple terms, it says that for a given locally compact, abelian group $G$, one can uniquely associate another locally compact, abelian group $\widehat{G}$, called the Pontryagin dual of $G$.
Moreover, this correspondence is dual in the sense that the dual group of $\widehat{G}$ is canonically isomorphic to $G$.

In the realm of operator algebras, Pontryagin duality paved the way to the development of a duality theory for actions of locally compact, abelian groups on operator algebras.
If $A$ is either a \cstar-algebra or a von Neumann algebra and $\alpha\colon G\curvearrowright A$ is an action of a locally compact, abelian group, then there exists a canonical action $\widehat{\alpha}\colon \widehat{G}\curvearrowright A\rtimes_\alpha G$ of the dual group on the crossed product.
In \cite{Takesaki}, Takesaki showed that this construction has a certain duality behaviour if $A$ is assumed to be a von Neumann algebra.
In fact, this duality result proved to be a quintessential ingredient in Connes' subdivision of type III factors \cite{Con73}.
On the \cstar-algebraic side, a similar result was later proved by Takai \cite{Takai}.
Furthermore, one might hope that the Takai--Takesaki duality not only provides nice structural results for crossed products, but in certain situations, the action $\alpha$ can be understood via its dual action $\widehat{\alpha}$ and vice versa.
More precisely, it is natural to ask to what extent one can relate properties of an action $\alpha$ to properties of its dual action $\widehat{\alpha}$. 

For the purpose of classification of group actions on \cstar-algebras, a property which has attracted an enormous amount of interest is the so-called Rokhlin property.
This is a condition which has been imported from a certain property of ergodic, measure-preserving actions of amenable groups on probability spaces \cite{Roh48, OW80}, and further defined for actions of a large class of abelian groups on operator algebras (see for example \cite{ConAut75,Kish77,BSKR93,KishRokhlinFlows,Izu04}).
Due to its considerable impact in the classification of \cstar-dynamics, a problem which has received significant attention in the last two decades is to characterise when a given action has the Rokhlin property.
In various cases, the Rokhlin property was shown to enjoy a duality-type relation to \emph{approximate representability}, as defined by Izumi in \cite{Izu04} in the setting of finite abelian groups.
Indeed, if the \cstar-algebra $A$ is separable, unital and the acting group $G$ is finite abelian, Izumi showed in \cite{Izu04} that the action $\alpha$ has the Rokhlin property if and only if the dual action $\widehat{\alpha}$ is approximately representable.
Moreover, $\alpha$ is approximately representable if and only if the dual action $\widehat{\alpha}$ has the Rokhlin property.
Further generalisations of this result for finite and compact groups can be found in \cite{GardSant16,Naw16,SeqSplit,GardellaRokhCpt,GardellaCtsRokh, GardellaCtsRokhCorrig}, while Kishimoto proved similar duality results for continuous actions of $\mathbb{R}$ on Kirchberg algebras \cite{Kish03,Kish04}.

Although the Rokhlin property has enjoyed a remarkable success in the classification of \cstar-dynamics, it is fair to say that it often imposes various $K$-theoretic conditions which can make it somewhat restrictive  and difficult to check in practice.
Roughly speaking, all kinds of Rokhlin-type properties for group actions on \cstar-algebras rely on finding finite (or compact \cite{AranoKubota17}) approximate representations of the acting group $G$ which are also approximately central.
In the present work, we study a property for actions of locally compact, abelian groups on \cstar-algebras called the \emph{abelian Rokhlin property}, which is a straightforward analogue of the way one defines the Rokhlin property for actions of abelian groups on von Neumann algebras \cite{Shimada14}.

\begin{defn}[Definition \ref{defn: AbelianRokhlin}]
Let $A$ be a separable \cstar-algebra and $G$ be a second-countable, locally compact, abelian group. Let $\alpha\colon G\curvearrowright A$ be a continuous action. Then we say that $\alpha$ has the abelian Rokhlin property if for any $\chi\in \widehat{G}$, there exists a unitary $u\in F_\infty(A)$ such that $\widetilde{\alpha}_{\infty,g}(u)=\chi(g)u$ for any $g\in G$.   
\end{defn}

In the case when the acting group is finite, this is a weakening of Izumi's Rokhlin property (see Remark \ref{rmk: AbRokFinite}), whereas in the case of the acting group being $\mathbb{R}^k$ for some $k\in\mathbb{N}$, it coincides with the known Rokhlin property.
With this definition in hand, we will provide a framework that unifies previous duality results appearing in the literature and show that the abelian Rokhlin property is dual to a property that we call \emph{pointwise strong approximate innerness} (Definition \ref{defn: ApproxInner}).
This coincides with Izumi's notion of strong approximate innerness when $G$ is finite.
Compared to approximate representability, we require that each automorphism belonging to the action is approximately inner via a sequence of unitaries that are approximate fixed points, but without asking the unitaries to form an approximate representation of the group $G$ (Remark \ref{rmk: IzumiApproxInner}).
The main result we prove is the following duality theorem obtained in the third section.
This result is inspired by and can be understood as a generalization of \cite[Theorem 4.6]{Kish04} to separable \cstar-algebras. 
Moreover, we note that if the acting group is $\mathbb{R}$, the following theorem is the \cstar-counterpart of a duality result proved by Masuda and Tomatsu in \cite[Theorem 4.11]{MasudaTomatsu}.

\begin{theorem}[Theorem \ref{thm: MainResultDuality}]
Let $A$ be a separable \cstar-algebra and $G$ a second-countable, locally compact, abelian group with dual group $\widehat{G}$.
Let $\alpha\colon G\curvearrowright A$ be a continuous action.
\begin{enumerate}[label=\textup{(\roman*)}]
    \item The action $\alpha$ is pointwise strongly approximately inner if and only if the dual $\widehat{G}$-action $\widehat{\alpha}$ has the abelian Rokhlin property.
    \item The action $\alpha$ has the abelian Rokhlin property if and only if the dual $\widehat{G}$-action $\widehat{\alpha}$ is pointwise strongly approximately inner.
\end{enumerate}    
\end{theorem}

In the last section, we show that the abelian Rokhlin property can be used to characterise densely defined, lower semicontinuous traces on crossed products.
For convenience we shall only state a special case here, but note that Theorem~\ref{thm: Traces} in the main body of the paper is stated with a strictly weaker assumption for actions of locally compact abelian groups.

\begin{theorem}\label{thm: ThmC}
Let $G$ be a countable discrete abelian group and $\alpha\colon G\curvearrowright A$ an action on a separable \cstar-algebra.
Suppose that $\alpha$ has the abelian Rokhlin property.
Then every densely defined lower semicontinuous trace on $A\rtimes_\alpha G$ factors through the canonical conditional expectation $A\rtimes_\alpha G\twoheadrightarrow A$.
\end{theorem}

At the end of the article we will comment how this result (or rather Theorem~\ref{thm: Traces}) can be used to unify the proofs of various similar results found in the literature, such as in \cite{Tho21,EST22,Nea24} (see Remark \ref{rmk: Traces}).


\subsection*{Acknowledgements}
RN and GS were supported by the European Research Council under the European Union's Horizon Europe research and innovation programme (ERC grant AMEN--101124789).
GS was furthermore supported by research project G085020N funded by the Research Foundation Flanders (FWO).
RN was furthermore supported by the postdoctoral fellowship 1204626N of the Research Foundation Flanders (FWO).
JC was supported by the postdoctoral fellowship 1291823N of the Research Foundation Flanders (FWO) and by a research grant (VIL72080) from Villum Fonden. 
We would like to thank the anonymous referee for the helpful comments and for suggesting an alternative proof of Theorem~\ref{thm: Traces}. 

For the purpose of open access, the authors have applied a CC BY public copyright license to any author accepted manuscript version arising from this submission.
\allowdisplaybreaks


\section{Preliminaries}
\numberwithin{theorem}{section}

\subsection{Sequence algebras}\label{subsect: SeqAlg}
In this subsection, we will record some standard facts about (central) sequence algebras.
We invite the reader to consult \cite{KirchbergAbel} for a more detailed account of the topic.

Given a separable \cstar-algebra $A$, we denote by $A_\infty=\ell^\infty(\mathbb{N},A)/c_0(\mathbb{N},A)$ the \cstar-algebra obtained as the quotient of the bounded sequences of $A$ by the null sequences.
In particular, we will realise $A$ as a subalgebra of $A_\infty$ by viewing elements in $A$ as constant sequences in $A_\infty$.
Let $D_{\infty,A}\coloneqq \overline{AA_\infty A}$ be the hereditary \cstar-subalgebra generated by $A$ in $A_\infty$.
Moreover, let
\[
\mathcal{N}(D_{\infty,A},A_\infty)=\{x\in A_\infty \mid xD_{\infty,A}+D_{\infty,A}x\subset D_{\infty,A}\}
\] 
be the normaliser of $D_{\infty,A}$ inside $A_\infty$ and 
\[
A_\infty\cap A^\perp=\{x\in A_\infty \mid \forall a\in A\colon xa=ax=0 \}
\] 
be the two-sided annihilator of $D_{\infty,A}$ inside $A_\infty$.
Then, by \cite[Proposition 1.9(4)+(5)]{KirchbergAbel} (see also \cite[Proposition 1.5]{SeqSplit}), we have a short exact sequence
\begin{equation}\label{eq: SESNorm}
\begin{tikzcd}
A_\infty\cap A^\perp \rightarrowtail \mathcal{N}(D_{\infty,A},A_\infty) \twoheadrightarrow  \mathcal{M}(D_{\infty,A}), 
\end{tikzcd}
\end{equation}
where the surjection is the canonical map induced by the universal property of the multiplier algebra.
Furthermore, we denote Kirchberg's central sequence algebra $(A_\infty\cap A')/(A_\infty\cap A^\perp)$ by $F_\infty(A)$. 

With the set-up above, let $G$ be a second-countable, locally compact group and $\alpha\colon G\curvearrowright A$ a continuous action.
Then we get a (possibly discontinuous) $G$-action $\alpha_\infty$ on $A_\infty$.
Since $D_{\infty,A}$ is an $\alpha_\infty$-invariant subalgebra, we get an action on $\mathcal{N}(D_{\infty,A},A_\infty)$ by restricting $\alpha_\infty$ and an induced action $\widetilde{\alpha}_\infty$ on $\mathcal{M}(D_{\infty,A})$.
Similarly, there is an induced action on $F_\infty(A)$ which is also denoted by $\widetilde{\alpha}_\infty$.
These actions are in general not point-norm continuous, so we restrict to the continuous part when necessary.
For instance, we consider
\[
A_{\infty,\alpha} = \{x \in A_\infty \mid [g \to \alpha_{\infty,g} (x)]\ \text{is norm-continuous}\}.
\]
Moreover, as a consequence of \cite[Theorem 2]{Brown00}, $A_{\infty,\alpha}$ coincides with the \cstar-algebra $\ell^\infty_\alpha(\mathbb{N},A)/c_0(\mathbb{N},A)$, where
\[
\ell^\infty_\alpha(\mathbb{N},A) = \{(x_n)_{n\in\mathbb{N}}\in \ell^\infty(\mathbb{N},A) \mid [g\to (\alpha_g (x_n))_{n\in\mathbb{N}}] \ \text{is continuous}\}.
\]
In particular, using \cite[Lemma 3.5]{SeqSplit}, we see for every $x\in A_{\infty,\alpha}$, any representing sequence $(x_n)_{n\in\mathbb{N}}\in \ell^\infty(\mathbb{N},A)$ of $x$, and every compact subset $K\subseteq G$ that 
\[
\max_{g\in K} \|\alpha_{g,\infty}(x)-x\|=\limsup_{n\to\infty}\max_{g\in K}\|\alpha_g(x_n)-x_n\|.
\]
We will often use this fact without mention throughout the paper.

\begin{rmk}\label{rmk: SigmaIdeal}
Evidently the maps in \eqref{eq: SESNorm} are all equivariant.
Since we can always find (see \cite[Lemma 1.4]{KA88}) a countable approximate unit of $A$, say $(e_n)_{n\in\mathbb{N}}$, such that $\lim_{n\to\infty}\max_{g\in K}\|\alpha_g(e_n)-e_n\|=0$ for every compact subset $K\subseteq G$ , it can be shown that the ideal in \eqref{eq: SESNorm} is an algebraic $G$-$\sigma$-ideal, as defined in \cite[Definition 4.1]{GSSSAII} if one treats $G$ as a discrete group.
Here we shall only need and justify a particular consequence of this fact, namely that the short exact sequence \eqref{eq: SESNorm} restricts to a short exact sequence of the involved fixed point algebras.
Given an element $x\in\mathcal{M}(D_{\infty,A})$ with $\widetilde{\alpha}_{\infty,g}(x)=x$ for all $g\in G$, we claim that it admits a lift to a fixed point in $\mathcal{N}(D_{\infty,A},A_\infty)$.
First we can lift it to an arbitrary element $y\in\mathcal{N}(D_{\infty,A},A_\infty)$.
Using a sequence $(e_n)_n$ as above, we see that $ye_n\in D_{\infty, A}$. In particular, this gives that $ye_n=xe_n$ and 
$\alpha_{\infty,g}(ye_n)=\widetilde{\alpha}_{\infty,g}(xe_n)=x\alpha_g(e_n)$ for all $n\in\mathbb{N}$ and $g\in G$.
As $(e_n)_n$ was an approximate unit consisting of approximate fixed points, we have $ye_n a\to ya$ and $aye_n\to ay$ for all $a\in A$ and $\|\alpha_{\infty,g}(ye_n)-ye_n\|\to 0$ uniformly over compact sets.
By applying a standard reindexation argument, this allows us to construct an element $z\in A_\infty$ with $z-y\in A_\infty\cap A^\perp$ (therefore $z\in\mathcal{N}(D_{\infty,A},A_\infty)$) and $\alpha_{\infty,g}(z)=z$ for all $g\in G$.
\end{rmk}

\subsection{Actions by abelian groups}\label{subsect: Actions}

Let $A$ be a separable \cstar-algebra, $G$ a second-countable, locally compact, abelian group, and $\alpha\colon G\curvearrowright A$ a continuous action.
We denote by $\iota_A$ the canonical inclusion of $A$ into $\mathcal{M}(A\rtimes_\alpha G)$ and by $\lambda^\alpha\colon G\to\mathcal{U}(\mathcal{M}(A\rtimes_\alpha G))$ the canonical unitary representation.
Note that if $G$ is discrete, then $\iota_A(A)\subseteq A\rtimes_\alpha G$.
Since $G$ is abelian, for any $g\in G$, we can define a $G$-action $\alpha^g\colon G\curvearrowright A\rtimes_{\alpha_g}\mathbb{Z}$ by
\[
\alpha_h^g(a\lambda_g^n)=\alpha_h(a)\lambda_g^n \quad\text{for } h\in G,\ a\in A,\ \text{and } n\in\mathbb{Z}.
\]
Here we denote the generating unitary of $A\rtimes_{\alpha_g}\mathbb{Z}$ by $\lambda_g$.


\section{The abelian Rokhlin property}

\begin{defn}\label{defn: AbelianRokhlin}
Let $A$ be a separable \cstar-algebra and $G$ a second-countable, locally compact, abelian group.
Let $\alpha\colon G\curvearrowright A$ be a continuous action.
Then we say that $\alpha$ has the abelian Rokhlin property if for any $\chi\in \widehat{G}$, there exists a unitary $u\in F_\infty(A)$ such that $\widetilde{\alpha}_{\infty,g}(u)=\chi(g)u$ for any $g\in G$.
Moreover, we say that $\alpha$ has the rational abelian Rokhlin property if $\alpha$ is equivariantly $\mathcal{Z}$-stable\footnote{Here $\mathcal{Z}$ stands for the Jiang-Su algebra introduced in \cite{JiangSu99}.} (that is $\alpha$ is cocycle conjugate to $\alpha\otimes\id_{\mathcal{Z}}$) and there exist distinct prime numbers $p,q$ such that $\alpha\otimes\id_{M_{p^\infty}}$ and $\alpha\otimes\id_{M_{q^\infty}}$ have the abelian Rokhlin property.
\end{defn}

We will now collect a series of remarks regarding the abelian Rokhlin property and its rational counterpart.

\begin{rmk}\label{rmk: AbRokFinite}
Let $\alpha\colon G\curvearrowright A$ be as above.
Given an element $\chi\in\widehat{G}$, we may define the action $\sigma^\chi\colon G\curvearrowright C(\mathbb{T})$ given by $\sigma^\chi_g(f)(z)=f(\chi(g)z)$ for any $g\in G$, $f\in C(\mathbb{T})$, and $z\in\mathbb{T}$. 
Then $\alpha$ has the abelian Rokhlin property if and only if for any $\chi\in\widehat{G}$, there exists a unital, equivariant $*$-homomorphism 
\[
\varphi_\chi\colon (C(\mathbb{T}),\sigma^\chi)\to (F_\infty(A),\widetilde{\alpha}_\infty).
\] 
As we comment in Remark \ref{rmk: IzumiApproxInner}, the maps $\varphi_\chi$ are not required to form a representation of the dual group $\widehat{G}$.
Assuming such an extra condition would be equivalent to assuming that the action $\alpha$ has the usual Rokhlin property for finite abelian groups (see for example the proof of \cite[Lemma 3.8]{Izu04}). 
\end{rmk}

\begin{rmk}\label{rmk: AbRokhIndepParam}
As a consequence of Lemma \ref{lemma: AbRokhSeqSplit} and Lemma \ref{lemma: RationalAbRokh}, the rational abelian Rokhlin property is in fact independent of the choice of distinct prime numbers $p,q$. 
Moreover, since any requirement on the structure of $\alpha\otimes \id_{M_{r^\infty}}$ is a property of the action $\alpha\otimes\id_{\mathcal{Z}}\otimes \id_{M_{r^\infty}}$, we have chosen to include equivariant $\mathcal{Z}$-stability as part of its definition.
This extra assumption does not cause any loss of generality in the results appearing in this article, as it can be seen from Theorem \ref{thm: AbRokhEquivDef} or Theorem \ref{thm: Traces}.
\end{rmk}

\begin{rmk}\label{rmk: AbRokhProp}
Let $A,B$ be separable \cstar-algebras and $\alpha\colon G\curvearrowright A$ and $\beta\colon G\curvearrowright B$ be two continuous actions of a second-countable, locally compact, abelian group $G$. 
Then, using the canonical inclusion $F_\infty(A)\subset F_\infty(A\otimes B)$, where $\otimes$ stands for either the minimal or the maximal tensor product, one can see that if $\alpha$ has the abelian Rokhlin property, then so does $\alpha\otimes\beta$.
\end{rmk}

\begin{rmk}\label{rmk: AbRokh}
In the case when $G=\mathbb{R}^k$, the abelian Rokhlin property is nothing but the known Rokhlin property.
In the case $k=1$, the Rokhlin property was defined by Kishimoto in the unital case \cite{KishRokhlinFlows} and extended to the nonunital case in \cite{SzaRokhlinFlows}.
The definition for $k\geq 2$ can be found in \cite[Definition 6.2]{GSRokhDimAbs}. When $k\geq 2$, the equivalence between the Rokhlin property and the abelian Rokhlin property follows from \cite[Proposition 6.5]{GSRokhDimAbs}.
\end{rmk}

When the group $G$ is the integer group, the abelian Rokhlin property has a different flavour compared to the finite group and multiflow cases.
We shall demonstrate that it is closely related to other dynamical properties such as finite Rokhlin dimension with commuting towers.
Let us first recall the definition of the latter for integer actions, which we identify with the single automorphisms generating them.
This notion was first introduced in \cite[Definition 2.3]{HWZ15} in the case when the \cstar-algebra is unital and extended to also cover the non-unital case in \cite[Definition 1.21]{HP15}.
In the particular case of $\mathbb{Z}$-actions, there are several different definitions for finite Rokhlin dimension depending on what kind of towers one allows.
Here, we will use the definition of finite Rokhlin dimension with commuting towers from \cite{SWZRokhDim19}, which one might informally refer to as ``Rokhlin dimension with single towers''.

\begin{defn}[{\cite[Definition 10.2]{SWZRokhDim19}}]\label{defn: RokhDim}
Let $A$ be a separable \cstar-algebra, $d \in \mathbb{N}$, and $n\in\mathbb{N}$.
An automorphism $\alpha$ of $A$ is said to have \emph{Rokhlin dimension $d$ with commuting towers of length $n$}, written $\dimrokc(\alpha,n\mathbb{Z})=d$, if $d$ is the least natural number such that there exist equivariant c.p.c.\ order zero maps
\[
\varphi_\ell\colon (C(\mathbb{Z}/n\mathbb{Z}), \mathbb{Z}\text{-shift})\to (F_\infty(A),\widetilde{\alpha}_\infty)
\]
for $\ell=0,\ldots,d$ with pairwise commuting ranges and such that
\[
1=\varphi_0(1)+\ldots+\varphi_d(1).
\]
We then define the Rokhlin dimension with commuting towers of $\alpha$ by 
\[
\dimrokc(\alpha)=\sup\limits_{n\in\mathbb{N}}\ \dimrokc(\alpha,n\mathbb{Z}).
\]
\end{defn}

We will need the following lemma.
Before proving it, we recall the notion of equivariantly sequentially split $*$-homomorphisms from \cite{SeqSplit}.

\begin{defn}[{\cite[Definition 3.3]{SeqSplit}}]\label{defn: SeqSplit}
Let $A$ and $B$ be separable \cstar-algebras and $G$ be a locally compact group.
Let $\alpha\colon G\curvearrowright A$ and $\beta\colon G \curvearrowright B$ be two continuous actions.
An equivariant $*$-homomorphism $\varphi\colon (A,\alpha) \to (B,\beta)$ is called ($G$-equivariantly) sequentially split, if there exists a commutative diagram of equivariant $*$-homomorphisms of the form
\[
\xymatrix{
(A,\alpha)\ar[dr]_{\varphi} \ar[rr] &&(A_{\infty,\alpha},\alpha_\infty), \\&(B,\beta)\ar[ru]&
}
\]
where the horizontal map is the canonical diagonal inclusion.   
\end{defn}

We use this framework as a technical gadget to perform certain proofs efficiently.
We will often use without further mention the fact that the composition of two equivariantly sequentially split maps is equivariantly sequentially split (see \cite[Proposition 3.7]{SeqSplit}), and the fact that if some composition $\psi\circ\varphi$ is equivariantly sequentially split, then so is $\varphi$; the latter is a direct consequence of the definition.
We shall also frequently use the following fact:

\begin{lemma}[see {\cite[Lemma 4.2]{SeqSplit}}] \label{lemma: central-ss}
Let $A$ and $C$ be separable \cstar-algebras and $G$ be a second-countable, locally compact group.
Suppose that $C$ is unital.
Let $\alpha\colon G\curvearrowright A$ and $\gamma\colon G \curvearrowright C$ be two continuous actions.
There exists an equivariant unital $*$-homomorphism from $(C,\gamma)$ to $(F_\infty(A),\widetilde{\alpha}_\infty)$ if and only if the first-factor embedding
\[
\id_A\otimes 1_C: (A,\alpha)\to (A\otimes_{\max} C,\alpha\otimes\gamma)
\]
is equivariantly sequentially split.
\end{lemma}

\begin{lemma}\label{lemma: RokhDimBound}
Let $A,B$ be separable \cstar-algebras, $\alpha\in\Aut(A)$, and $\beta\in\Aut(B)$.
Suppose that $B$ is unital and that there exists an equivariant unital $*$-homomorphism from $(B,\beta)$ to $(F_\infty(A),\widetilde{\alpha}_\infty)$.
Then $\dimrokc(\alpha)\leq \dimrokc(\beta)$.
\end{lemma}

\begin{proof}
Denote such a $*$-homomorphism by $\theta$.
Let $\dimrokc(\beta)=d$ for some $d\in\mathbb{N}$. Then for each $n\in\mathbb{N}$, there exist equivariant c.p.c.\ order zero maps 
\[
\varphi_0^{(n)},\ldots,\varphi_d^{(n)}\colon (C(\mathbb{Z}/n\mathbb{Z}),\mathbb{Z}\text{-shift})\to (B_\infty\cap B',\beta_\infty)
\]
with pairwise commuting ranges such that $\varphi_0^{(n)}(1)+\ldots+\varphi_d^{(n)}(1)=1$.

Let $C$ be the \cstar-algebra generated by the images of the maps $\varphi_j^{(n)}$ for all $0\leq j\leq d$ and all $n\in\mathbb N$.
Clearly $C$ is $\beta_\infty$-invariant and we denote $\gamma=\beta_\infty|_C\in\Aut(C)$.
Then the map 
\[
\id_B\otimes 1_C\colon (B,\beta)\to (B\otimes_{\max}C,\beta\otimes\gamma)
\]
is equivariantly sequentially split by Lemma~\ref{lemma: central-ss}.
Since $\id_A$ is (trivially) equivariantly sequentially split, so is $\id_A\otimes\id_B\otimes 1_C$ by \cite[Theorem 2.8(IV)]{SeqSplit} (note that exactly the same proof works in the equivariant setting).

Moreover, another application of Lemma~\ref{lemma: central-ss} shows that the map $\id_A\otimes 1_B$ is equivariantly sequentially split.
Hence, the composition $(\id_A\otimes\id_B\otimes 1_C)\circ (\id_A\otimes 1_B)$ is equivariantly sequentially split.
As 
\begin{equation}\label{eq: RokhDimSeqSplit}
\begin{array}{ccl}
(\id_A\otimes\id_B\otimes 1_C)\circ (\id_A\otimes 1_B) &=& \id_A\otimes 1_B\otimes 1_C \\
&=& (\id_A\otimes 1_B\otimes \id_C)\circ (\id_A\otimes 1_C), 
\end{array}   
\end{equation}
we get that $(\id_A\otimes 1_C)$ is equivariantly sequentially split.
Hence there exists a unital and equivariant $*$-homomorphism $(C,\gamma)\to (F_\infty(A),\widetilde{\alpha}_\infty)$ by Lemma~\ref{lemma: central-ss}.
Since $C$ was generated by systems of Rokhlin towers of arbitrary length, this shows that $\dimrokc(\alpha)\leq d$.
\end{proof}

\begin{prop}\label{prop: RokhDim}
Let $A$ be a separable \cstar-algebra and $\alpha\in\Aut(A)$ a single automorphism with the abelian Rokhlin property.
Then $\dimrokc(\alpha)\leq 1$.   
\end{prop}

\begin{proof}
Let $\theta\in\mathbb{R}\setminus\mathbb Q$ and set $\chi=e^{2\pi i\theta}\in\mathbb{T}$.
Since $\alpha$ has the abelian Rokhlin property, there exists a unital equivariant $*$-homomorphism $\varphi\colon(C(\mathbb{T}),\sigma^\chi)\to (F_\infty(A),\widetilde{\alpha}_\infty)$.
Moreover, $\dimrokc(\sigma^\chi)=1$ by \cite[Theorem 6.2]{HWZ15}. Hence, $\dimrokc(\alpha)\leq 1$ by Lemma \ref{lemma: RokhDimBound}.
\end{proof}

One can obtain a partial converse to Proposition \ref{prop: RokhDim}.
Before proving this converse, we will first characterise the abelian Rokhlin property using the notion of equivariantly sequentially split morphisms.

\begin{lemma} \label{lemma: AbRokhSeqSplit}
Let $A$ be a separable \cstar-algebra and $G$ be a second-countable, locally compact, abelian group.
Let $\alpha\colon G\curvearrowright A$ be a continuous action.
The action $\alpha$ has the abelian Rokhlin property if and only if for any $\chi\in\widehat{G}$, the equivariant $*$-homomorphism
\[
\id_A\otimes 1\colon (A,\alpha)\hookrightarrow (A\otimes C(\mathbb{T}),\alpha\otimes\sigma^\chi)
\]
is $G$-equivariantly sequentially split, where $\sigma^\chi$ is the $G$-action given by $\sigma_g^\chi(f)(z)=f(\chi(g)z)$ for any $g\in G$, $f\in C(\mathbb{T})$, $z\in\mathbb{T}$.
\end{lemma}

\begin{proof}
The statement follows by combining Lemma~\ref{lemma: central-ss} and Remark \ref{rmk: AbRokFinite}.    
\end{proof}

Next we require a preparatory lemma that we prove in the more general setting of second-countable, locally compact groups, since it may be of independent interest in that level of generality.
We say that a separable \cstar-algebra $A$ is \emph{approximately divisible} if there exists a unital $*$-homomorphism $\psi\colon M_2\oplus M_3\to F_\infty(A)$.
A formally stronger notion with the same name was introduced in \cite[Definition 1.2]{BKR92} when $A$ is unital and extended to the non-unital case in \cite[Definition 5.5]{KR00}. 
The following is analogous to an observation in the latter reference, the proof of which is standard but is included for the reader's convenience:

\begin{lemma} \label{lemma: approx-div}
Let $A$ be a separable \cstar-algebra.
Let $p,q\geq 2$ be two natural numbers with $\operatorname{gcd}(p,q)=1$.
The following are equivalent
\begin{enumerate}[label=\textup{(\roman*)}]
\item $A$ is approximately divisible. \label{lemma: approx-div:1}
\item There exists a unital $*$-homomorphism $M_p\oplus M_q\to F_\infty(A)$. \label{lemma: approx-div:2}
\item There exists a unital $*$-homomorphism $(M_p\oplus M_q)^{\otimes\infty}\to F_\infty(A)$. \label{lemma: approx-div:3}
\end{enumerate}
\end{lemma}
\begin{proof}
We note that the equivalence \ref{lemma: approx-div:2}$\Leftrightarrow$\ref{lemma: approx-div:3} is a result of \cite[Corollary 1.13]{KirchbergAbel}.
Two such natural numbers must allow an expression $p=2a+3b$ and $q=2c+3d$ for some natural numbers $a,b,c,d\geq 0$.
This induces a unital $*$-homomorphism $M_2\oplus M_3\to M_p\oplus M_q$, which immediately yields the implication \ref{lemma: approx-div:2}$\Rightarrow$\ref{lemma: approx-div:1}.

Conversely, assume that $A$ is approximately divisible.
Since we already know the equivalence \ref{lemma: approx-div:2}$\Leftrightarrow$\ref{lemma: approx-div:3} for arbitrary pairs $(p,q)$, we can apply it to deduce that there exists a unital $*$-homomorphism $(M_2\oplus M_3)^{\otimes\infty}\to F_\infty(A)$.
Clearly it suffices to argue that there exists a unital $*$-homomorphism $M_p\oplus M_q\to (M_2\oplus M_3)^{\otimes\infty}$.
For this, since $p$ and $q$ are coprime, we note that every natural number $N>pq$ has an expression $N=ap+bq$ for some natural numbers $a,b\geq 0$.
Let $\ell\geq 1$ be big enough that $2^{\ell}>pq$.
This ensures that for all $k\in\{0,\dots,\ell\}$, there exists a unital $*$-homomorphism $M_p\oplus M_q\to M_{2^k 3^{\ell-k}}$.
Hence, we also find a unital $*$-homomorphism from $M_p\oplus M_q$ into
\[
\bigoplus_{k=0}^\ell M_{2^k 3^{\ell-k}}^{\oplus{\ell\choose k}} \cong (M_2\oplus M_3)^{\otimes\ell} \subset (M_2\oplus M_3)^{\otimes\infty}.
\]
Here we used the binomial theorem for tensor powers of direct sums.
\end{proof}

\begin{lemma}\label{lemma: RationalAbRokh}
Let $G$ be a second-countable and locally compact group, let $A, B$ be separable \cstar-algebras and $\alpha\colon G\curvearrowright A$, $\beta\colon G\curvearrowright B$ continuous actions.
Let $\varphi\colon (A,\alpha)\to (B,\beta)$ be an equivariant $*$-homomorphism.
Then the following are equivalent.
\begin{enumerate}[label=\textup{(\roman*)}]
\item There exist two distinct prime numbers $p$ and $q$ such that $\varphi\otimes\id_{M_{p^\infty}}$ and $\varphi\otimes\id_{M_{q^\infty}}$ are equivariantly sequentially split, where we equip $M_{p^\infty}$ and $M_{q^\infty}$ with the corresponding trivial action.\label{item: RationalAbRokh}
\item For every separable, approximately divisible \cstar-algebra $C$, the $*$-homomorphism $\varphi\otimes\id_C$ is equivariantly sequentially split, where we equip $C$ with the trivial action.\label{item: ApproxDiv}
\end{enumerate}
\end{lemma}

\begin{proof}
The implication \ref{item: ApproxDiv}$\Rightarrow$\ref{item: RationalAbRokh} is trivial since $M_{p^\infty}$ is approximately divisible for any prime number $p$.
We will show that \ref{item: RationalAbRokh} implies \ref{item: ApproxDiv}. 

We first claim that it suffices to show that $\varphi\otimes\id_E$ is equivariantly sequentially split, where $E=(M_p\oplus M_q)^{\otimes\infty}$.
Let $C$ be a separable, approximately divisible \cstar-algebra.
By Lemma~\ref{lemma: approx-div}, there exists a unital $*$-homomorphism from $E$ to $F_\infty(C)$, so the first-factor embedding $\id_C\otimes 1_E$ is sequentially split by Lemma~\ref{lemma: central-ss}.
As we equip both $C$ and $E$ with the trivial $G$-actions, we have that $\id_A\otimes\id_C\otimes 1_E$ is equivariantly sequentially split by \cite[Theorem 2.8(IV)]{SeqSplit}.
So if we assume that $\varphi\otimes\id_E$ is equivariantly sequentially split, so is $\varphi\otimes\id_C\otimes\id_E$ by \cite[Theorem 2.8(IV)]{SeqSplit}.
Hence, the composition
\[
(\varphi\otimes\id_C\otimes\id_E)\circ(\id_A\otimes\id_C\otimes 1_E)=(\id_B\otimes\id_C\otimes 1_E)\circ(\varphi\otimes\id_C)
\]
is equivariantly sequentially split and thus also $\varphi\otimes\id_C$.

It remains to show that \ref{item: RationalAbRokh} implies that $\varphi\otimes\id_E$ is equivariantly sequentially split.
Let $\varepsilon>0$, choose finite sets $\mathcal{F}\Subset A$ and $ \mathcal{G}\Subset B$ with $\varphi(\mathcal{F})\subseteq \mathcal{G}$, and let $K\subseteq G$ be a compact set.
Let $r\in\{p,q\}$.
Since $\varphi\otimes\id_{M_{r^\infty}}$ is equivariantly sequentially split, there exists an equivariant $*$-homomorphism from $(B,\beta)$ to $(A\otimes M_{r^\infty})_\infty$ that sends $\varphi(a)$ to $a\otimes 1$ for all $a\in A$.
If we represent such $*$-homomorphisms via sequences of $*$-linear maps\footnote{Note that one can first lift the $*$-linear map from $B$ to $(A\otimes M_{r^\infty})_\infty$ to a linear map $\psi_0\colon B\to \ell^\infty(A\otimes M_{r^\infty})$. 
We then get a $*$-linear lift by defining $\psi(b)=\frac{1}{2}(\psi_0(b)+\psi_0(b^*)^*)$.}, we may find two $*$-linear maps $\psi_p\colon B\to A\otimes M_{p^\infty}$ and $\psi_q\colon B\to A\otimes M_{q^\infty}$ such that for $r\in\{p,q\}$, one has
\begin{itemize}
	\item $\|\psi_r(b)\|\leq\|b\|+\varepsilon$ for all $b\in\mathcal G$;
    \item $\|\psi_r(b_1b_2)-\psi_r(b_1)\psi_r(b_2)\|\leq\varepsilon$ for all $b_1,b_2\in\mathcal{G}$;
    \item $\max\limits_{g\in K}\|(\alpha_g\otimes\id_{M_{r^\infty}})(\psi_r(b))-\psi_r(\beta_g(b))\|\leq\varepsilon$ for all $b\in\mathcal{G}$;
    \item $\|\psi_r(\varphi(a))-a\otimes 1_{M_{r^\infty}}\|\leq\varepsilon$ for all $a\in\mathcal{F}$.
\end{itemize}
Since $\mathcal{G}\subseteq B$ is finite, we can assume there exists $\ell\in\mathbb{N}$ such that $\psi_r(\mathcal{G})\subseteq A\otimes M_{r^\ell}$ for each $r\in\{p,q\}$.
By composing $\psi_r$ with a conditional expectation $A\otimes M_{r^\infty}\to A\otimes M_{r^\ell}$ at the expense of a small perturbation over this finite set, we can and will assume that $\psi_r(B)\subseteq A\otimes M_{r^\ell}$ for $r\in\{p,q\}$, while all the four conditions above are still satisfied.

By the binomial theorem for tensor powers of direct sums, we have that
\[
(M_p\oplus M_q)^{\otimes 2\ell}\cong \bigoplus\limits_{j=0}^{2\ell} (M_{p^j}\otimes M_{q^{2\ell-j}})^{\oplus\binom {2\ell}j}.
\]
We see that each of these direct summands decomposes as a tensor product where one of the factors is either $M_{p^\ell}$ or $M_{q^\ell}$.
Thus, modulo composing with appropriate maps that move the tensor copy of $A$ to the left-most side of the tensor product, we obtain a $*$-linear map $\psi\colon B\to A\otimes (M_p\oplus M_q)^{\otimes 2\ell}\subseteq A\otimes E$ via
\[
\psi=\bigoplus\limits_{j=0}^\ell (1_{p^j}\otimes \psi_q\otimes 1_{q^{\ell-j}})^{\oplus \binom {2\ell}{j}}\oplus\bigoplus\limits_{j=\ell+1}^{2\ell}(\psi_p\otimes 1_{p^{j-\ell}}\otimes 1_{q^{2\ell-j}})^{\oplus \binom {2\ell}{j}}.
\]
Based on the properties of $\psi_p$ and $\psi_q$, we can immediately observe that 
\begin{itemize}
	\item $\|\psi(b)\|\leq\|b\|+\varepsilon$ for all $b\in\mathcal G$;
    \item $\|\psi(b_1b_2)-\psi(b_1)\psi(b_2)\|\leq\varepsilon$ for all $b_1,b_2\in\mathcal{G}$;
    \item $\max\limits_{g\in K}\|(\alpha_g\otimes\id_{E})(\psi(b))-\psi(\beta_g(b))\|\leq\varepsilon$ for all $b\in\mathcal{G}$;
    \item $\|\psi(\varphi(a))-a\otimes 1_{E}\|\leq\varepsilon$ for all $a\in\mathcal{F}$.
\end{itemize}
Let $\mathcal E\Subset E$ be a finite subset and let $\sigma$ be the one-sided Bernoulli shift endomorphism on $E=(M_p\oplus M_q)^{\otimes\infty}$ given by the formula \[\sigma(x_1\otimes x_2\otimes\ldots)=1\otimes x_1\otimes x_2\otimes\ldots \ .\]
If we compose $\psi$ with $\id_A\otimes\sigma^N$ for $N$ large enough, then all of the above properties are preserved and we may further assume that
\[
\|[\psi(b),1_{\mathcal{M}(A)}\otimes x]\|\leq\varepsilon \quad\text{for all } b\in\mathcal{G} \text{ and } x\in\mathcal E.
\] 
Since the quintuple $(\varepsilon,\mathcal{F},\mathcal{G},K,\mathcal{E})$ was arbitrary, we can thus find an equivariant $*$-homomorphism 
\[
\kappa: (B,\beta)\to \big( (A\otimes E)_\infty\cap (1_{\mathcal{M}(A)}\otimes E)', (\alpha\otimes\id_E)_\infty \big) \quad\text{with}\quad \kappa\circ\varphi=\id_A\otimes 1_E.
\]
By the universal property of the tensor product (keep in mind that $E$ is nuclear), this induces an equivariant $*$-homomorphism 
\[
\theta: (B\otimes E,\beta\otimes\id_E)\to \big( (A\otimes E)_\infty, (\alpha\otimes\id_E)_\infty \big) \quad\text{with}\quad \theta\circ(\varphi\otimes\id_E)=\id_A\otimes\id_E.
\]
This finishes the proof.
\end{proof}

As a consequence of the lemma above, we can obtain an a priori weaker characterisation for the rational abelian Rokhlin property.

\begin{prop}\label{prop: RationalAbRokh}
Let $G$ be a second-countable, locally compact, abelian group, let $A$ be a separable \cstar-algebra and $\alpha\colon G\curvearrowright A$. Then the following statements are equivalent: 
\begin{enumerate}[label=\textup{(\roman*)}]
\item $\alpha$ has the rational abelian Rokhlin property;\label{item: AbRokh}
\item $\alpha$ is cocycle conjugate to $\alpha\otimes\id_{\mathcal Z}$ and $\alpha\otimes\id_C$ has the abelian Rokhlin property for every separable, approximately divisible \cstar-algebra $C$.\label{item: AbRokhApproxDiv}
\end{enumerate}
\end{prop}

\begin{proof}
Using Lemma~\ref{lemma: AbRokhSeqSplit}, we can deduce the equivalence by applying Lemma~\ref{lemma: RationalAbRokh} to each of the equivariant maps $\id_A\otimes 1\colon (A,\alpha)\to (A\otimes C(\mathbb{T}),\alpha\otimes\sigma^\chi)$, where $\chi\in\widehat{G}$ is arbitrarily chosen.
\end{proof}

\begin{theorem}\label{thm: AbRokhEquivDef}
Let $A$ be a separable \cstar-algebra such that $A\cong A\otimes\mathcal{Z}$. If $\alpha\colon\mathbb{Z}\curvearrowright A$,
then the following statements are equivalent:
\begin{enumerate}[label=\textup{(\roman*)}]
\item $\dimrokc(\alpha)<\infty$;\label{item: RokhDim}
\item $\alpha$ is cocycle conjugate to $\alpha\otimes\gamma$ for some strongly outer action $\gamma\colon \mathbb{Z}\curvearrowright\mathcal{Z}$;\label{item: Absorb}
\item $\alpha$ is cocycle conjugate to $\alpha\otimes\gamma$ for every automorphism $\gamma$ of $\mathcal{Z}$;\label{item: Absorb-all}
\item $\alpha$ has the rational abelian Rokhlin property.\label{item: AbRokh1}
\end{enumerate}
\end{theorem}

\begin{proof}
The equivalence \ref{item: Absorb}$\Leftrightarrow$\ref{item: Absorb-all} is a direct consequence of \cite[Corollary 3.5]{GSSSA19}.

For the rest of the proof, let us fix a strongly outer action $\gamma\colon \mathbb{Z}\curvearrowright\mathcal{Z}$.
Note that such an action $\gamma$ has finite Rokhlin dimension with commuting towers by \cite[Theorem 2.14]{GSSSA19}.
Furthermore such an action $\gamma$ is automatically strongly self-absorbing by \cite[Theorem 3.2]{GSSSA19}.\footnote{We note that in these earlier articles, the third-named author used the outdated terminology of ``semi-strongly self-absorbing'' actions; see \cite[Definition 1.4]{GSSSA19} for the original definition and see \cite[Section 5]{Szabo21} for why the terminology has been revised.}
Hence, we get the implication \ref{item: Absorb}$\Rightarrow$\ref{item: RokhDim} as $\dimrokc(\alpha\otimes\gamma)\leq\dimrokc(\gamma)$\footnote{Similarly to Remark \ref{rmk: AbRokhProp}, this follows using the canonical inclusion $F_\infty(\mathcal{Z})\subset F_\infty(A\otimes\mathcal{Z})$.}, and the implication \ref{item: RokhDim}$\Rightarrow$\ref{item: Absorb} is a consequence of \cite[Corollary B]{GSRokhDimAbs}.

We shall argue the implication \ref{item: AbRokh1}$\Rightarrow$\ref{item: Absorb}.
Let $p$ and $q$ be two distinct prime numbers for which the automorphisms $\alpha\otimes\id_{M_{p^\infty}}$ and $\alpha\otimes\id_{M_{q^\infty}}$ have the abelian Rokhlin property.
Proposition~\ref{prop: RokhDim} implies that these automorphisms have finite Rokhlin dimension with commuting towers.
Applying \cite[Corollary B]{GSRokhDimAbs} again, it follows for $r\in\{p, q\}$ that $\alpha\otimes\id_{M_{r^\infty}}$ is cocycle conjugate to $\alpha\otimes\id_{M_{r^\infty}}\otimes\gamma$.
By \cite[Theorem 6.6(ii)]{Szabo17ssa3}, we get that $\alpha\otimes\id_{\mathcal Z}$ is cocycle conjugate to $\alpha\otimes\id_{\mathcal Z}\otimes\gamma$.
Since we assumed that $\alpha$ was cocycle conjugate to $\alpha\otimes\id_{\mathcal Z}$, this shows that $\alpha$ is cocycle conjugate to $\alpha\otimes\gamma$.
 
Lastly, let us argue the implication \ref{item: Absorb-all}$\Rightarrow$\ref{item: AbRokh1}.
Clearly we get that $\alpha$ is cocycle conjugate to $\alpha\otimes\id_{\mathcal Z}$.
As above, we still keep the choice of $\gamma$.
Let $p$ and $q$ be two distinct prime numbers.
For $r\in\{p,q\}$, we then have that $\alpha\otimes\id_{M_{r^\infty}}$ is cocycle conjugate to $\alpha\otimes\id_{M_{r^\infty}}\otimes\gamma$.
Using Remark \ref{rmk: AbRokhProp}, it suffices to show that the automorphism $\gamma\otimes\id_{M_{r^\infty}}$ has the abelian Rokhlin property.
Since $M_{r^\infty}\otimes\mathcal{Z}\cong M_{r^\infty}$, this automorphism is conjugate to some strongly outer action $\delta\colon \mathbb{Z}\curvearrowright M_{r^\infty}$.
The fact that such an action has the abelian Rokhlin property is a direct consequence of \cite[Theorem 1.3]{Kishimoto95} and \cite[Theorem 1.6]{BratteliEvansKishimoto95}.
Namely, the latter reference allows us to find, for any $\lambda\in\mathbb{T}$, a unitary $u\in (M_{r^\infty})_\infty\cap (M_{r^\infty})'$ such that $\bar{\lambda}=u \delta_\infty(u)^*$, or equivalently $\delta_\infty(u)=\lambda u$.
This finishes the proof.
\end{proof}


\section{Dual properties of group actions}

In this section, we will prove that the abelian Rokhlin property is dual to a property which we will call being pointwise strongly approximately inner.
From this we will deduce the corresponding duality result for actions with the rational abelian Rokhlin property (Definition \ref{defn: AbelianRokhlin}).

\begin{defn}\label{defn: ApproxInner}
Let $A$ be a separable \cstar-algebra and $G$ be a second-countable, locally compact, abelian group.
Let $\alpha\colon G\curvearrowright A$ be a continuous action.
We say that $\alpha$ is pointwise strongly approximately inner if for every $g\in G$, there exists a sequence of contractions $v_n$ in $A$ such that 
\begin{itemize}
    \item $v_nv_n^*$ and $v_n^*v_n$ converge strictly to $1_{\mathcal{M}(A)}$;
    \item $\alpha_g(a)=\lim\limits_{n\to\infty}v_nav_n^*$ for all $a\in A$;
    \item $\lim\limits_{n\to\infty}\max\limits_{h\in K}\|\alpha_h(v_n)-v_n\|=0$ for every compact subset $K\subseteq G$.
\end{itemize}
In the spirit of Definition \ref{defn: AbelianRokhlin}, we say that $\alpha$ is rationally pointwise strongly approximately inner if $\alpha$ is cocycle conjugate to $\alpha\otimes\id_{\mathcal{Z}}$ and there exist two distinct prime numbers $p$ and $q$ such that $\alpha\otimes\id_{M_{p^\infty}}$ and $\alpha\otimes\id_{M_{q^\infty}}$ are pointwise strongly approximately inner.
\end{defn}

\begin{rmk}\label{rmk: IzumiApproxInner}
In the case when $A$ is unital and the group $G$ is finite abelian, the definition above coincides with Izumi's notion of a strongly approximately inner action, which is a weakening of approximate representability (see \cite[Definition 3.6(1)]{Izu04}).
In particular, unlike in the case of approximate representability, we do not require the given unitaries to form a representation of the group.
\end{rmk}

To prove that pointwise strong approximate innerness and the abelian Rokhlin property are dual to each other, we shall express pointwise strong approximate innerness in the framework of equivariantly sequentially split $*$-homomorphisms, similar to how we did for the abelian Rokhlin property in the previous section.
We invite the reader to recall Definition \ref{defn: SeqSplit} as well as the notational convention from Subsection~\ref{subsect: Actions}.

\begin{lemma}\label{lemma: EquivSeqSplit}
Let $A$ be a separable \cstar-algebra and $G$ a second-countable, locally compact, abelian group.
Let $\alpha\colon G\curvearrowright A$ be a continuous action.
The action $\alpha$ is pointwise strongly approximately inner if and only if for every $g\in G$, the canonical equivariant inclusion
\[
\iota_g\colon (A,\alpha)\hookrightarrow (A\rtimes_{\alpha_g}\mathbb{Z}, \alpha^g)
\]
is $G$-equivariantly sequentially split.
\end{lemma}

\begin{proof}
Suppose that $\alpha$ is pointwise strongly approximately inner and fix $g\in G$.
Then there exists  a sequence of contractions $v_n$ in $A$ such that $v_nv_n^*$ and $v_n^*v_n$ converge strictly to $1_{\mathcal{M}(A)}$, $\alpha_g(a)=\lim\limits_{n\to\infty}v_nav_n^*$ for any $a\in A$, and $\lim\limits_{n\to\infty}\max\limits_{h\in K}\|\alpha_h(v_n)-v_n\|=0$ for every compact subset $K\subseteq G$.
Then let $v=[(v_n)_{n\in\mathbb{N}}]\in (A_\infty)^{\alpha_\infty}$ be the induced element in the sequence algebra, which is clearly fixed by the induced action $\alpha_\infty$.
Then $v^*v$ and $vv^*$ act like units upon multiplying with elements in $A$.
Since it follows for all $a\in A$ that 
\[
va=vav^*v=\alpha_g(a)v \quad\text{and}\quad av=\alpha_g(\alpha_g^{-1}(a))v=v\alpha_g^{-1}(a)v^*v=v\alpha_g^{-1}(a),
\]
it follows that $v\in\mathcal{N}(D_{\infty,A},A_\infty)$.
We get that $v$ induces a unitary $\widetilde{v}\in\mathcal{U}(\mathcal{M}(D_{\infty,A}))$ by the canonical surjection in \eqref{eq: SESNorm}, which is clearly fixed under the induced action $\widetilde{\alpha}_\infty$.
If we denote by $\lambda_g$ the canonical unitary in $\mathcal{M}(A\rtimes_{\alpha_g}\mathbb{Z})$, then we may define an equivariant $*$-homomorphism 
\[
\varphi\colon (A\rtimes_{\alpha_g}\mathbb{Z},\alpha^g)\to (A_{\infty,\alpha},\alpha_\infty) \quad\text{by}\quad \varphi(a\lambda_g^n)=a\widetilde{v}^n\quad\text{for } a\in A,\ n\in\mathbb{Z}.
\]
By construction, the composition $\varphi\circ\iota_g\colon A\to A_\infty$ is the canonical diagonal embedding, so $\iota_g$ is $G$-equivariantly sequentially split.

Conversely, fix $g\in G$ and suppose that $\iota_g$ is $G$-equivariantly sequentially split.
Let $\varphi\colon (A\rtimes_{\alpha_g}\mathbb{Z},\alpha^g)\to (A_{\infty,\alpha},\alpha_\infty)$ be an equivariant $*$-homomorphism such that $\varphi\circ\iota_g$ is the canonical diagonal inclusion.
The diagonal inclusion $A\hookrightarrow A_\infty$ maps into $D_{\infty,A}$ and is in that sense nondegenerate.
Hence, $\varphi$ maps into $D_{\infty,A}$ and there exists a unital $*$-homomorphism 
\[
\widetilde{\varphi}\colon\mathcal{M}(A\rtimes_{\alpha_g}\mathbb{Z})\to\mathcal{M}(D_{\infty,A})
\] 
extending $\varphi$.
By naturality, this map is equivariant with respect to the actions $\alpha^g$ and $\widetilde{\alpha}_\infty$.
Set $\widetilde{v}=\widetilde{\varphi}(\lambda_g)$, which is a unitary in $\mathcal{M}(D_{\infty,A})$ fixed by $\widetilde{\alpha}_\infty$.
By Remark~\ref{rmk: SigmaIdeal}, there exists a contraction $v=[(v_n)_{n\in\mathbb{N}}]\in \mathcal{N}(D_{\infty,A},A_\infty)^{\alpha_\infty}$ lifting $\widetilde{v}$ under the map in \eqref{eq: SESNorm}.
The fact that $\widetilde{v}$ is a unitary precisely yields that both sequences $v_nv_n^*$ and $v_n^*v_n$ converge strictly to $1_{\mathcal{M}(A)}$.
Furthermore, $\alpha_g(a)=\varphi(\alpha_g(a))$ and $\alpha_g(a)=\lambda_ga\lambda_g^*$, so 
\[
\alpha_g(a)=\varphi(\lambda_ga\lambda_g^*)=vav^* =\lim_{n\to\infty} v_n a v_n^*
\]
for any $a\in A$.
Finally, as remarked in Subsection \ref{subsect: SeqAlg}, the fact that $v$ is fixed under $\alpha_\infty$ translates to $\lim\limits_{n\to\infty}\max\limits_{h\in K}\|\alpha_h(v_n)-v_n\|=0$ for every compact subset $K\subseteq G$.
As $g\in G$ was arbitrary, this shows that $\alpha$ is pointwise strongly approximately inner.
\end{proof}

\begin{lemma}\label{lemma: CommDiagrams}
Let $A$ be a separable \cstar-algebra and $G$ be a second-countable, locally compact, abelian group. Let $\alpha\colon G\curvearrowright A$ be a continuous action.
\begin{enumerate}[label=\textup{(\roman*)}]
    \item For any $g\in G$, let $\sigma^g\colon\widehat{G}\curvearrowright C(\mathbb{T})$ be defined by $\sigma^g_\chi(f)(z)=f(\chi(g)z)$ for any $\chi\in\widehat{G}$, $f\in C(\mathbb{T})$, and $z\in\mathbb{T}$. Then there exists an equivariant isomorphism
    \[
    \eta_g\colon((A\rtimes_\alpha G)\otimes C(\mathbb{T}),\widehat{\alpha}\otimes\sigma^g)\to ((A\rtimes_{\alpha_g}\mathbb{Z})\rtimes_{\alpha^g}G,\widehat{\alpha^g})
    \]
    such that the diagram
\[
\xymatrix{
(A\rtimes_\alpha G,\widehat{\alpha})\ar[dr]_{\id_{A\rtimes_\alpha G}\otimes 1~~~} \ar[rr]^{\iota_g\rtimes G} &&((A\rtimes_{\alpha_g}\mathbb{Z})\rtimes_{\alpha^g}G,\widehat{\alpha^g}) \\&((A\rtimes_\alpha G)\otimes C(\mathbb{T}),\widehat{\alpha}\otimes\sigma^g)\ar[ru]_{\eta_g}&
}
\]
commutes.\label{item: CommDiag1}
    \item For any $\chi\in \widehat{G}$, there exists an equivariant isomorphism
    \[
    \zeta_\chi\colon ((A\rtimes_\alpha G)\rtimes_{\widehat{\alpha}_\chi}\mathbb{Z},\widehat{\alpha}^\chi)\to ((A\otimes C(\mathbb{T}))\rtimes_{\alpha\otimes\sigma^\chi} G,\widehat{\alpha\otimes\sigma^\chi})
    \]
such that the diagram
    \[
\xymatrix{
(A\rtimes_\alpha G,\widehat{\alpha})\ar[dr] \ar[rr]^{(\id_A\otimes 1)\rtimes G} &&((A\otimes C(\mathbb{T}))\rtimes_{\alpha\otimes\sigma^\chi}G,\widehat{\alpha\otimes\sigma^\chi}) \\&((A\rtimes_\alpha G)\rtimes_{\widehat{\alpha}_\chi}\mathbb{Z},\widehat{\alpha}^\chi)\ar[ru]_{\zeta_\chi}&
}
\]
commutes, where the downward map is the canonical inclusion.\label{item: CommDiag2}
\end{enumerate}
\end{lemma}

\begin{proof}
\ref{item: CommDiag1}:
Fix $g\in G$ and let $w=\id_{\mathbb T}\in C(\mathbb{T})$.
Define an equivariant $*$-homomorphism 
\[
j_g\colon (C(\mathbb{T}),\sigma^g)\to(\mathcal{M}((A\rtimes_{\alpha_g}\mathbb{Z})\rtimes_{\alpha^g}G),\widehat{\alpha^g})
\]
by $j_g(w)=\lambda_g^*\lambda_g^{\alpha^{g}}$, where $\lambda_g$ is the canonical unitary associated to $A\rtimes_{\alpha_g}\mathbb{Z}$ and $\lambda_g^{\alpha^g}$ is one of the canonical unitaries in $\mathcal{M}((A\rtimes_{\alpha_g}\mathbb{Z})\rtimes_{\alpha^g}G)$.
Note that we still denote by $\widehat{\alpha^g}$ the induced action on $\mathcal{M}((A\rtimes_{\alpha_g}\mathbb{Z})\rtimes_{\alpha^g}G)$.
Furthermore, using that $G$ is abelian, we also have that the image of $j_g$ commutes with the image of $\iota_g\rtimes G$, so by the universal property of the maximal tensor product, there exists an equivariant $*$-homomorphism $\eta_g= (\iota_g\rtimes G)\otimes j_g\colon((A\rtimes_\alpha G)\otimes C(\mathbb{T}),\widehat{\alpha}\otimes\sigma^g)\to ((A\rtimes_{\alpha_g}\mathbb{Z})\rtimes_{\alpha^g}G,\widehat{\alpha^g})$.
Moreover, it is immediate to see that $\eta_g$ makes the required diagram commute.

To show that $\eta_g$ is an isomorphism, we will build its inverse. First, the nondegenerate $*$-homomorphism $\iota_A\otimes 1\colon A\to \mathcal{M}((A\rtimes_\alpha G)\otimes C(\mathbb{T}))$ and the unitary representation
\[
\mathbb{Z}\to \mathcal{U}(\mathcal{M}((A\rtimes_\alpha G)\otimes C(\mathbb{T}))), \ 1\mapsto \lambda_g^\alpha\otimes w^*
\]
form a covariant pair for $(A,\alpha_g)$, so they induce a $*$-homomorphism 
\[
\xi_g\colon A\rtimes_{\alpha_g}\mathbb{Z}\to \mathcal{M}((A\rtimes_\alpha G)\otimes C(\mathbb{T})).
\]
Then consider the unitary representation 
\[
G\to \mathcal{U}(\mathcal{M}((A\rtimes_\alpha G)\otimes C(\mathbb{T}))),\ h\mapsto \lambda_h^\alpha\otimes 1.
\]
Again, since $G$ is abelian, this unitary representation together with $\xi_g$ form a covariant pair for $(A\rtimes_{\alpha_g}\mathbb{Z},\alpha^g)$.
Indeed, for every $h\in G$, $a\in A$ and $n\in\mathbb Z$, we have that 
\begin{align*}
    (\lambda_h^\alpha\otimes 1)\xi_g(a\lambda_g^n)((\lambda_h^\alpha)^*\otimes 1) &= (\lambda_h^\alpha\otimes 1)(\iota_A(a)(\lambda_g^\alpha)^n\otimes w^{*n})((\lambda_h^\alpha)^*\otimes 1)\\ &= \iota_A(\alpha_h(a))(\lambda_g^\alpha\otimes w^*)^n\\
    &= \xi_g(\alpha^g_h(a\lambda_g^n)).
\end{align*}
Therefore, we obtain a $*$-homomorphism
\[
\mu_g\colon (A\rtimes_{\alpha_g}\mathbb{Z})\rtimes_{\alpha^g}G\to(A\rtimes_\alpha G)\otimes C(\mathbb{T}),
\]
which is the inverse of $\eta_g$.
This finishes the proof of \ref{item: CommDiag1}.

\ref{item: CommDiag2}:
Fix $\chi\in\widehat{G}$ and still denote $w=\id_{\mathbb T}\in C(\mathbb{T})$.
The $*$-homomorphism 
\begin{equation}\label{eq: CovHom1}
(\id_A\otimes 1)\rtimes G\colon A\rtimes_\alpha G\to (A\otimes C(\mathbb{T}))\rtimes_{\alpha\otimes\sigma^\chi}G
\end{equation}
and the unitary representation
\[
\mathbb{Z}\to \mathcal{U}(\mathcal{M}((A\otimes C(\mathbb{T}))\rtimes_{\alpha\otimes\sigma^\chi}G))
\]
sending $1$ to $\iota_{A\otimes C(\mathbb{T})}(1_{\mathcal{M}(A)}\otimes w^*)$ define a covariant pair for $(A\rtimes_\alpha G, \widehat{\alpha}_\chi)$.
Indeed, $A\otimes C(\mathbb{T})$ commutes with $\iota_{A\otimes C(\mathbb{T})}(1_{\mathcal{M}(A)}\otimes w^*)$ and for every $g\in G$ we have
\begin{align*}
&\iota_{A\otimes C(\mathbb{T})}(1_{\mathcal{M}(A)}\otimes w^*)\lambda_g^{\alpha\otimes\sigma^\chi}\iota_{A\otimes C(\mathbb{T})}(1_{\mathcal{M}(A)}\otimes w) \\
&= \iota_{A\otimes C(\mathbb{T})}(1_{\mathcal{M}(A)}\otimes w^*)\iota_{A\otimes C(\mathbb{T})}(1_{\mathcal{M}(A)}\otimes\chi(g) w)     \lambda_g^{\alpha\otimes\sigma^\chi}\\
&= \chi(g)\lambda_g^{\alpha\otimes\sigma^\chi}.
\end{align*}
Then let $\zeta_\chi\colon (A\rtimes_\alpha G)\rtimes_{\widehat{\alpha}_\chi}\mathbb{Z}\to (A\otimes C(\mathbb{T}))\rtimes_{\alpha\otimes\sigma^\chi}G$ be the induced $*$-homomorphism.
The left side carries the action $\widehat{\alpha}^\chi$ and the right side carries the action $\widehat{\alpha\otimes\sigma^\chi}$.
Clearly $\zeta_\chi$ is already equivariant on $A\rtimes_\alpha G$ and by construction it sends the canonical unitary $\lambda_\chi$ of $(A\rtimes_\alpha G)\rtimes_{\widehat{\alpha}_\chi}\mathbb{Z}$, which is fixed under $\widehat{\alpha}^\chi$, to $\iota_{A\otimes C(\mathbb{T})}(1_{\mathcal{M}(A)}\otimes w^*)$, which is likewise fixed under $\widehat{\alpha\otimes\sigma^\chi}$.
Thus $\zeta_\chi$ is equivariant. 
Note that $\zeta_\chi$ makes the required diagram commute by \eqref{eq: CovHom1}.

To construct the inverse of $\zeta_\chi$, note that the $*$-homomorphism
\[
A\otimes C(\mathbb{T})\to \mathcal{M}((A\rtimes_\alpha G)\rtimes_{\widehat{\alpha}_\chi}\mathbb{Z})
\]
given by sending $a\otimes w$ to $\iota_A(a)\lambda_\chi^*$ and the unitary representation
\[
G\to \mathcal{U}(\mathcal{M}((A\rtimes_\alpha G)\rtimes_{\widehat{\alpha}_\chi}\mathbb{Z}))
\]
given by $g\mapsto \lambda_g^\alpha$ define a covariant pair for $(A\otimes C(\mathbb{T}),\alpha\otimes\sigma^\chi)$.
Here $\lambda_\chi$ stands for the canonical unitary of $(A\rtimes_\alpha G)\rtimes_{\widehat{\alpha}_\chi}\mathbb{Z}$.
To see this, observe that for all $a\in A$, $g\in G$ and $n\in\mathbb Z$, one has 
\begin{align*}
    \lambda_g^\alpha\iota_A(a)(\lambda_\chi^n)^* (\lambda_g^\alpha)^* &= \iota_A(\alpha_g(a))\lambda_g^\alpha(\lambda_{\chi}^n)^*(\lambda_g^\alpha)^*\\
    &= \iota_A(\alpha_g(a)) \lambda_g^\alpha \widehat{\alpha}_{\chi}^{-n}(\lambda_g^\alpha)^*(\lambda_{\chi}^n)^* \\ &= \iota_A(\alpha_g(a))\chi(g)^n(\lambda_\chi^n)^*.
\end{align*}
Thus, the induced $*$-homomorphism $(A\otimes C(\mathbb{T}))\rtimes_{\alpha\otimes\sigma^\chi}G\to (A\rtimes_\alpha G)\rtimes_{\widehat{\alpha}_\chi}\mathbb{Z}$ is the inverse of $\zeta_\chi$. 
\end{proof}

To obtain the rational version of the duality result below, we need to collect one more fact about equivariant $\mathcal{Z}$-stability. 
Since the proposition below works in the generality of \emph{strongly self-absorbing} \cstar-algebras (\cite[Definition 1.3(iv)]{TomsWinterSSA}), we will prove it in this generality. Note that if $\mathcal{D}$ is a strongly self-absorbing \cstar-algebra, an action $\alpha$ is \emph{equivariantly $\mathcal{D}$-stable} if $\alpha$ is cocycle conjugate to $\alpha\otimes\id_{\mathcal{D}}$.

\begin{rmk}\label{rmk: TensorCrossedProd}
Let $G$ be a second-countable, locally compact, abelian group, $A$ be a \cstar-algebra and $\alpha\colon G\curvearrowright A$. Note that for any \cstar-algebra $C$, we have canonical $\widehat{G}$-equivariant isomorphisms 
\[
\big( (A\rtimes_\alpha G)\otimes C, \widehat{\alpha}\otimes\id_C\big) \cong \big( (A\otimes C)\rtimes_{\alpha\otimes \id_C}G, \widehat{\alpha\otimes\id_C} \big).
\] 
\end{rmk}

\begin{prop}\label{prop: EquivZStab}
Let $A$ be a separable \cstar-algebra and $G$ be a second-countable, locally compact, abelian group.
Let $\alpha\colon G\curvearrowright A$ be a continuous action and $\mathcal{D}$ be a strongly self-absorbing \cstar-algebra.
Then $\alpha$ is equivariantly $\mathcal{D}$-stable if and only if $\widehat{\alpha}$ is equivariantly $\mathcal{D}$-stable.   
\end{prop}

\begin{proof}
Combining \cite[Corollary 3.8]{Szabo18ssa} and \cite[Lemma 4.2]{SeqSplit}, $\alpha$ is equivariantly $\mathcal{D}$-stable if and only if the first-factor embedding $\id_A\otimes 1_{\mathcal{D}} \colon (A, \alpha) \to
(A \otimes \mathcal{D}, \alpha\otimes\id)$ is equivariantly sequentially split.
Then, combining \cite[Corollary 3.17]{SeqSplit} and Remark \ref{rmk: TensorCrossedProd}, $\id_A\otimes 1_{\mathcal{D}}$ is equivariantly sequentially split if and only if the $*$-homomorphism \[(\id_A\rtimes G)\otimes 1_{\mathcal{D}}\colon (A\rtimes_\alpha G,\widehat{\alpha})\to ((A\rtimes_\alpha G)\otimes\mathcal{D},\widehat{\alpha}\otimes\id)\] is equivariantly sequentially split. 
Using again \cite[Corollary 3.8]{Szabo18ssa} and \cite[Lemma 4.2]{SeqSplit}, $(\id_A\rtimes G)\otimes 1_{\mathcal{D}}$ is equivariantly sequentially split if and only if $\widehat{\alpha}$ is equivariantly $\mathcal{D}$-stable. 
\end{proof}

\begin{theorem}\label{thm: MainResultDuality}
Let $A$ be a separable \cstar-algebra and $G$ be a second-countable, locally compact, abelian group.
Let $\alpha\colon G\curvearrowright A$ be a continuous action.
\begin{enumerate}[label=\textup{(\roman*)}]
    \item The action $\alpha$ is (rationally) pointwise strongly approximately inner if and only if the dual action $\widehat{\alpha}$ has the (rational) abelian Rokhlin property.\label{item: MainThm1}
    \item The action $\alpha$ has the (rational) abelian Rokhlin property if and only if the action $\widehat{\alpha}$ is (rationally) pointwise strongly approximately inner.\label{item: MainThm2}
\end{enumerate}
\end{theorem}

\begin{proof}
Let us first consider the statements concerning the abelian Rokhlin property and pointwise strong approximate innerness.

\ref{item: MainThm1}:
By Lemma \ref{lemma: EquivSeqSplit}, $\alpha$ is pointwise strongly approximately inner if and only if the canonical inclusion
\[
\iota_g\colon (A,\alpha)\hookrightarrow (A\rtimes_{\alpha_g}\mathbb{Z}, \alpha^g)
\]
is $G$-equivariantly sequentially split for every $g\in G$.
Then, \cite[Corollary 3.17]{SeqSplit} shows that this happens if and only if the dual map 
\[
\iota_g\rtimes G\colon (A\rtimes_\alpha G,\widehat{\alpha})\to ((A\rtimes_{\alpha_g}\mathbb{Z})\rtimes_{\alpha^g}G,\widehat{\alpha^g})
\]
is $\widehat{G}$-equivariantly sequentially split for every $g\in G$.
By Lemma~\ref{lemma: CommDiagrams}\ref{item: CommDiag1}, this is equivalent to the map 
\[
\id_{A\rtimes_\alpha G}\otimes 1\colon (A\rtimes_\alpha G,\widehat{\alpha})\to ((A\rtimes_\alpha G)\otimes C(\mathbb{T}),\widehat{\alpha}\otimes\sigma^g)
\]
being $\widehat{G}$-equivariantly sequentially split for every $g\in G$.
By Lemma~\ref{lemma: AbRokhSeqSplit}, this is indeed equivalent to $\widehat{\alpha}$ having the abelian Rokhlin property.

\ref{item: MainThm2}:
By Lemma \ref{lemma: AbRokhSeqSplit}, $\alpha$ has the abelian Rokhlin property if and only if for every $\chi\in\widehat{G}$, the equivariant $*$-homomorphism 
\[
\id_A\otimes 1\colon (A,\alpha)\hookrightarrow (A\otimes C(\mathbb{T}),\alpha\otimes\sigma^\chi)
\]
is $G$-equivariantly sequentially split. 
Then, \cite[Corollary 3.17]{SeqSplit} shows that this is the case if and only if the dual map 
\[
(\id_A\otimes 1)\rtimes G\colon (A\rtimes_\alpha G,\widehat{\alpha})\to ((A\otimes C(\mathbb{T}))\rtimes_{\alpha\otimes\sigma^\chi}G, \widehat{\alpha\otimes\sigma^\chi})
\]
is $\widehat{G}$-equivariantly sequentially split for every $\chi\in\widehat{G}$. 
By Lemma~\ref{lemma: CommDiagrams}\ref{item: CommDiag2}, this is equivalent to the canonical inclusion 
\[
(A\rtimes_\alpha G,\widehat{\alpha})\to ((A\rtimes_\alpha G)\rtimes_{\widehat{\alpha}_\chi}\mathbb{Z},\widehat{\alpha}^\chi)
\]
being $\widehat{G}$-equivariantly sequentially split for every $\chi\in\widehat{G}$.
By Lemma~\ref{lemma: EquivSeqSplit}, this is indeed equivalent to $\widehat{\alpha}$ being pointwise strongly approximately inner.

We now prove the statements with the rational versions of the abelian Rokhlin property and pointwise strong approximate innerness.
By definition, $\alpha$ is rationally pointwise strongly approximately inner if and only if $\alpha$ is equivariantly $\mathcal{Z}$-stable and there exist distinct prime numbers $p,q$ such that $\alpha\otimes \id_{M_{p^\infty}}$ and $\alpha\otimes \id_{M_{q^\infty}}$ are pointwise strongly approximately inner.
Combining Proposition \ref{prop: EquivZStab}, the proof of \ref{item: MainThm1} above, and the identification of $\widehat{\alpha}\otimes\id_C$ with $\widehat{\alpha\otimes\id_C}$ from Remark \ref{rmk: TensorCrossedProd}, this happens if and only if $\widehat{\alpha}$ has the rational abelian Rokhlin property. 
The rational version of the statement in \ref{item: MainThm2} follows by the analogous argument.
\end{proof}

Combining Theorem~\ref{thm: MainResultDuality} and Remark~\ref{rmk: AbRokh} yields the following corollary that generalises a similar duality result for flows proved by Kishimoto in the case when $A$ is a unital UCT Kirchberg algebra (see \cite[Theorem 4.6]{Kish04} or \cite[Theorem 1.3]{Kish03}).
In these references, the pointwise strong approximate innerness for flows was described as ``each $\alpha_t$ is $\alpha$-invariantly approximately inner'' and just defined for unital \cstar-algebras.

\begin{cor}\label{cor: Flows}
Let $A$ be a separable \cstar-algebra, $k\in\mathbb{N}$, and $\alpha\colon\mathbb{R}^k\curvearrowright A$ be a continuous action.
\begin{enumerate}[label=\textup{(\roman*)}]
    \item The action $\alpha$ is pointwise strongly approximately inner if and only if the dual action $\widehat{\alpha}$ has the Rokhlin property.\label{item: Flows1}
    \item The action $\alpha$ has the Rokhlin property if and only if the action $\widehat{\alpha}$ is pointwise strongly approximately inner.\label{item: Flows2}
\end{enumerate}
\end{cor}


\section{Traces on crossed products}

In this section, we will show that the abelian Rokhlin property allows one to obtain a full description of the densely defined lower semicontinuous traces on the associated crossed products as those coming from invariant traces.
Further below we will argue that this argument unifies various special results of this type in the literature that were obtained with comparably more ad-hoc arguments.

\begin{notation}
Given a \cstar-algebra $A$, we write $\tilde{T}(A)$ for the cone of densely defined lower semicontinuous tracial weights on $A$.
If we are given an action $\alpha: G\curvearrowright A$ of some group, then $\tilde{T}(A)^\alpha$ denotes the set of $\alpha$-invariant traces, i.e., those traces $\tau$ with $\tau=\tau\circ\alpha_g$ for all $g\in G$.
\end{notation}

\begin{rmk}\label{rmk: CPTrace}
Let $\alpha\colon G\curvearrowright A$ be a continuous action of a locally compact group on a \cstar-algebra.
The crossed product $A\rtimes_\alpha G$ is a suitable norm-closure of $C_{\mathrm{c}}(G,A)$, which carries the $*$-algebra structure given by a convolution product.
Given an invariant trace $\tau\in\tilde{T}(A)^\alpha$, the assignment
\[
\hat{\tau}: C_{\mathrm{c}}(G,A)_+\to [0,\infty],\quad \hat{\tau}(f)=\tau(f(0)),
\]
extends to a densely defined lower semicontinuous trace $\hat{\tau}\in\tilde{T}(A\rtimes_\alpha G)$.
\end{rmk}

The following is well-known. 
To our knowledge, this result goes back to \cite[Theorem 5.1]{VigandPedersen82}. 
This requires a bit of translation in the language of this article. 
Given a \cstar-dynamical system $(A,G,\alpha)$, we apply \cite[Theorem 5.1]{VigandPedersen82} in the case when the multiplier $\chi$ is $1$.
Then, note that a densely defined lower semicontinuous tracial weight is automatically semifinite, which yields that the $(G,1)$-traces in \cite[Theorem 5.1]{VigandPedersen82} are precisely the densely defined lower semicontinuous $\alpha$-invariant tracial weights on $A$. 
Moreover, a densely defined lower semicontinuous weight satisfying the KMS condition for the trivial flow at inverse temperature $1$ is a tracial weight, so the densely defined $\widehat{\alpha}$-invariant semitraces with multiplier $1$ are exactly the densely defined lower semicontinuous $\widehat{\alpha}$-invariant tracial weights.
With these identifications, one gets the following statement.

\begin{prop}[{\cite[Theorem 5.1]{VigandPedersen82}}]\label{prop: tracial-correspondence}
Let $\alpha\colon G\curvearrowright A$ be a continuous action of a locally compact, abelian group on a \cstar-algebra.
Then
\[
\{ \hat{\tau} \mid \tau\in\tilde{T}(A)^\alpha \} = \tilde{T}(A\rtimes_\alpha G)^{\widehat{\alpha}}.
\]
\end{prop}

Combining Proposition~\ref{prop: tracial-correspondence} and Theorem~\ref{thm: MainResultDuality}, we obtain a complete description of densely defined lower semicontinuous traces on crossed products.

\begin{theorem}\label{thm: Traces}
Let $A$ be a separable \cstar-algebra, $G$ be a second-countable, locally compact, abelian group and $\alpha\colon G\curvearrowright A$ an action.
Suppose that there exists a separable \cstar-algebra $C$ with a unique tracial state $\tau_0$ and with $\tilde{T}(C)=\{\lambda \tau_{0}\mid \lambda\in\mathbb{R}_+\}$ such that $\alpha\otimes\id_C$ has the abelian Rokhlin property (e.g., if $\alpha$ has the rational abelian Rokhlin property).\footnote{Here $\otimes$ stands for the minimal tensor product.} 
Then 
\[
\tilde{T}(A\rtimes_\alpha G)=\{\hat{\tau} \mid \tau\in\tilde{T}(A)^\alpha \}.
\]
\end{theorem}
\begin{proof}
On the one hand, we have a canonical isomorphism $(A\otimes C)\rtimes_{\alpha\otimes\id_C} G\cong (A\rtimes_\alpha G)\otimes C$.
Due to the assumptions on $C$, if $\tau_0$ is the unique tracial state on $C$, then the assignment $\tau\mapsto\tau\otimes\tau_0$ induces a bijection $\tilde{T}(B)\to\tilde{T}(B\otimes C)$ for any \cstar-algebra $B$.
Combining these facts, we get a commutative diagram of maps
\[
\xymatrix{
\tilde{T}(A)^\alpha \ar[rr]^{\tau\mapsto\hat{\tau}} \ar[d] && \tilde{T}(A\rtimes_\alpha G) \ar[d] \\
\tilde{T}(A\otimes C)^{\alpha\otimes\id_C} \ar[rr]  && \tilde{T}\big( (A\otimes C)\rtimes_{\alpha\otimes\id_C} G \big)
}
\]
where the vertical maps are bijections.
Our claim amounts to the assertion that the upper horizontal arrow is surjective, which is hence clearly equivalent to saying that the lower horizontal arrow is surjective.
By replacing $\alpha\colon G\curvearrowright A$ with $\alpha\otimes\id_C\colon G\curvearrowright A\otimes C$, we may therefore assume without loss of generality that $\alpha$ has the abelian Rokhlin property.

By Proposition~\ref{prop: tracial-correspondence}, it suffices to prove that every trace $\theta\in\tilde{T}(A\rtimes_\alpha G)$ is $\widehat{\alpha}$-invariant.
Let us fix a trace $\theta$, $\chi\in\widehat{G}$, and $x\in (A\rtimes_\alpha G)_+$.
We shall argue that $\theta(x)=\theta(\widehat{\alpha}_\chi(x))$.
Since $\alpha$ has the abelian Rokhlin property, Theorem~\ref{thm: MainResultDuality} implies that there exists a sequence of contractions $(v_n)_{n\in\mathbb{N}}$ such that $v_n^*v_n$ converges strictly to the unit in the multiplier algebra of $A\rtimes_\alpha G$ and $\widehat{\alpha}_\chi(x)=\lim_{n\to\infty}v_nxv_n^*$. 
Then, by lower semicontinuity of $\theta$ we get that
\begin{align*}
\theta(\widehat{\alpha}_\chi(x)) &= \theta(\lim_{n\to\infty}v_nxv_n^*)\leq \liminf_{n\to\infty}\theta(v_nx^{1/2}x^{1/2}v_n^*) = \liminf_{n\to\infty}\theta(x^{1/2}v_n^*v_nx^{1/2})\\&\leq \liminf_{n\to\infty}\theta(x) = \theta(x).
\end{align*}
Moreover, running the same argument with $\chi^{-1}$ and $\widehat{\alpha}_\chi(x)$ instead of $\chi$ and $x$ shows that 
\[
\theta(x)=\theta(\widehat{\alpha}_{\chi^{-1}}(\widehat{\alpha}_\chi(x)))\leq \theta(\widehat{\alpha}_\chi(x)),
\] and hence $\theta(x)=\theta(\widehat{\alpha}_\chi(x))$.
As $\theta$, $\chi$, and $x$ were arbitrary, this finishes the proof.
\end{proof}

\begin{rmk}\label{rmk: Traces}
First, note that Theorem \ref{thm: ThmC} is a special case of the theorem above. 
To see this, take $C$ to be the complex numbers and note that when $G$ is discrete, the formula in Remark \ref{rmk: CPTrace} becomes $\hat{\tau}=\tau\circ E$ for any invariant trace $\tau\in \tilde{T}(A)^\alpha$, where $E\colon A\rtimes_\alpha G\to A$ is the canonical conditional expectation.
Moreover, Theorem \ref{thm: Traces} provides a unifying framework for proving many results of this nature found in the literature.
In the case when the acting group is $\mathbb{Z}$, the results in \cite[Lemma 3.6(A)]{EST22} and \cite[Lemma 3.4(b)]{Tho21}, or the application of \cite[Lemma 2.17]{Nea24} in \cite[Lemma 3.7]{Nea24} and \cite[Lemma 4.5]{Nea24}, can all be recovered by using Theorem~\ref{thm: Traces} instead.
To see this, note that in all those situations, the $\mathbb Z$-actions are given as the tensor product of some automorphism $\alpha_0$ with a strongly outer automorphism of $\mathcal{Z}$; typically one chooses the noncommutative Bernoulli shift via the identification $\mathcal{Z}\cong\mathcal{Z}^{\otimes\mathbb Z}$.
By Theorem \ref{thm: AbRokhEquivDef}, automorphisms constructed in this fashion have the rational abelian Rokhlin property and therefore the desired description of the traces on the crossed product follows.
\end{rmk}

\end{document}